\documentclass[a4paper,11pt]{article}

\usepackage{amssymb,amsmath,amsthm}
\usepackage{setspace}
\usepackage{slashbox}
\usepackage{graphicx} 
\usepackage{enumerate}
\usepackage{chngcntr}
\usepackage{mathtools}
\counterwithin{equation}{section}
\usepackage{marginnote}
\usepackage{pgf,tikz}
\usepackage{mathrsfs}
\usetikzlibrary{arrows}
\usetikzlibrary[patterns]
\tikzstyle{dashed}=[dash pattern=on 3pt off 3pt]

\usepackage{tikz-cd}
\usetikzlibrary{%
  matrix,%
  calc,%
  arrows%
}

\usepackage{verbatim}

\usepackage[backend=bibtex, style=ieee-alphabetic,maxbibnames=99]{biblatex}

\usepackage{filecontents}
\usepackage{sectsty}
\usepackage{lipsum}

\theoremstyle{definition}

\newtheorem{thm}{Theorem}[section]
\newtheorem{prop}[thm]{Proposition}
\newtheorem{coro}[thm]{Corollary}
\newtheorem{lem}[thm]{Lemma}
\newtheorem{dfn}[thm]{Definition}

\newtheorem{con}[thm]{Condition}
\newtheorem{ex}[thm]{Example}

\newtheorem{rmk}[thm]{Remark}

\allsectionsfont{\centering}

\providecommand{\lan}{\mathcal{L}}
\providecommand{\M}{\mathcal{M}}
\providecommand{\cl}{\mathcal{C}}

\newcommand{\rel}{\mathbf{r}}

\def\Ind#1#2{#1\setbox0=\hbox{$#1x$}\kern\wd0\hbox to 0pt{\hss$#1\mid$\hss}
\lower.9\ht0\hbox to 0pt{\hss$#1\smile$\hss}\kern\wd0}
\def\Notind#1#2{#1\setbox0=\hbox{$#1x$}\kern\wd0\hbox to 0pt{\mathchardef
\nn="3236\hss$#1\nn$\kern1.4\wd0\hss}\hbox to 0pt{\hss$#1\mid$\hss}\lower.9\ht0
\hbox to 0pt{\hss$#1\smile$\hss}\kern\wd0}
\def\ind{\mathop{\mathpalette\Ind{}}}

\begin{document}

\title{Simplicity of the automorphism groups of some binary homogeneous structures determined by triangle constraints}

\author{Yibei Li}

\date{}

\maketitle

\section*{Abstract}

We study some amalgamation classes introduced by Cherlin in the appendix of \cite{cherlin1998classification} and prove the simplicity of the automorphism groups of the Fra{\"\i}ss{\'e} limits of these classes. We employ the machinery of stationary independence relations used by Tent and Ziegler in \cite{tentziegler2013isometry}.

\section{Introduction}

\subsection{Overview}

Given a relational language $\lan$, a countable $\lan$-structure $\M$ is \emph{homogeneous} if every partial isomorphism between finite substructures of $\M$ extends to an automorphism of $\M$. Some examples include the random graph, the generic $k$-uniform hypergraph, the rationals with linear order etc. Fra{\"\i}ss{\'e}'s Theorem \cite{fraisse1953theorem} provides one way of constructing homogeneous structures by establishing a one-to-one correspondence between such structures and \emph{amalgamation classes} (see Definition \ref{amalgamationclass}). We call the homogeneous structure the \emph{Fra{\"\i}ss{\'e} limit} of the corresponding amalgamation class. For example, the random graph is the Fra{\"\i}ss{\'e} limit of the set of all finite graphs, the Urysohn space is the completion of the Fra{\"\i}ss{\'e} limit of the set of all finite rational-valued metric spaces.

Given finite $\lan$-structures $A,B,C$ where $B \subseteq A,C$, the \emph{free amalgam} of $A,C$ over $B$ is the $\lan$-structure $D$ on the disjoint union of $A,C$ over $B$ and for each relation $R \in \lan$, $R^D=R^A \cup R^C$. An amalgamation class $\cl$ is \emph{free} if it is closed under taking free amalgams, i.e. for $A,B,C \in \cl$ such that $B$ can be embedded in $A,C$, the free amalgam of $A,C$ over $B$ is also in $\cl$. A homogeneous structure is \emph{free} if it is the Fra{\"\i}ss{\'e} limit of a free amalgamation class. The following theorem about free homogeneous structures was proved by Macpherson and Tent \cite{macphersontent2011simplicity} using ideas and results from model theory and topological groups:

\begin{thm}\label{tztheorem}(\cite{macphersontent2011simplicity})
Let $\M$ be a countable free homogeneous relational structure. Suppose $Aut(\M) \neq Sym(\M)$ and $Aut(\M)$ is transitive on $\M$. Then $Aut(\M)$ is simple.
\end{thm}

Tent and Ziegler \cite{tentziegler2013isometry} generalised the theorem to a weaker notion of free homogeneous structures, namely a homogeneous structure with a stationary independence relation. They applied their method to the Urysohn space, which is not free, but has a local stationary independence relation (see Definition \ref{sirdfn}). They also used their approach to show that the isometry group of the bounded Urysohn space is simple in \cite{tentziegler2013bounded}. In this paper, we apply their method to the Fra{\"\i}ss{\'e} limits of all amalgamation classes given by Cherlin in the appendix of \cite{cherlin1998classification}. Cherlin's classes satisfy a weaker property, called \emph{semi-free amalgamation}, defined as follows:

\begin{dfn}
Given a relational language $\lan$, let $\cl$ be an amalgamation class of finite $\lan$-structures. We say $\mathcal{C}$ is a \emph{semi-free amalgamation class} if there exists $\lan ' \subsetneq \lan$ such that for any finite structures $A,B,C \in \mathcal{C}$ and embeddings $f_1:B \rightarrow A, f_2:B\rightarrow C$, there exist $D\in \mathcal{C}$ and embeddings $g_1:A \rightarrow D, g_2:C\rightarrow D$ such that $g_1f_1(B)=g_2f_2(B)=g_1(A) \cap g_2(C)$ and for any $a \in g_1(A)\setminus g_1 f_1(B), c\in g_2(C)\setminus g_2 f_2(B)$, if $a,c$ are related by some $R \in \lan$, then $R \in \lan'$. We call $\lan'$ the \emph{set of solutions}.
\end{dfn}

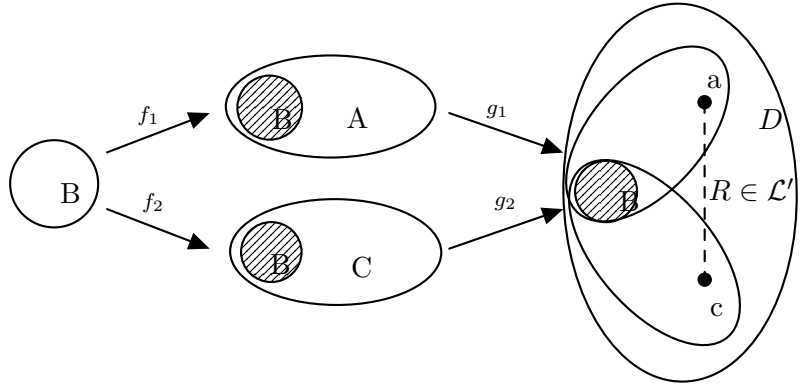
\begin{center}
\begin{tikzpicture}[line cap=round,line join=round,>=triangle 45,x=0.7499999999999991cm,y=0.7492917847025493cm]
\clip(-5.3228,-2.3684) rectangle (11.9172,4.6516);
\draw [line width=0.8pt] (-2.428,1.448) ellipse (0.5788134414472413cm and 0.578266875402438cm);
\draw [->,line width=0.8pt] (-1.5,2.004) -- (0.3,2.7);
\draw [->,line width=0.8pt] (-1.5,1.) -- (0.3,0.3);
\draw [rotate around={0.:(2.42,2.816)},line width=0.8pt] (2.42,2.816) ellipse (1.378102486419107cm and 0.6769783619889579cm);
\draw [line width=0.8pt,fill=black,pattern=north east lines,pattern color=black] (1.352,2.796) ellipse (0.4253895273746164cm and 0.42498783754706754cm);
\draw [line width=0.8pt,fill=black,pattern=north east lines,pattern color=black] (1.38,0.24) ellipse (0.3976697750646883cm and 0.39729426064064266cm);
\draw [rotate around={-0.4305174793771619:(2.509,0.24)},line width=0.8pt] (2.509,0.24) ellipse (1.3884888711422085cm and 0.7016297092862699cm);
\draw [rotate around={49.1861783969235:(7.9766,2.3298)},line width=0.8pt] (7.9766,2.3298) ellipse (1.3894778106444732cm and 0.7470362260149427cm);
\draw [rotate around={-49.343421910637666:(8.0986,0.2308)},line width=0.8pt] (8.0986,0.2308) ellipse (1.4778506880218971cm and 0.7639351833433878cm);
\draw [line width=0.8pt,fill=black,pattern=north east lines,pattern color=black] (7.25,1.312) ellipse (0.4074039764165288cm and 0.4070192701120754cm);
\draw [->,line width=0.8pt] (4.5,2.7) -- (6.5,2.);
\draw [->,line width=0.8pt] (4.5,0.3) -- (6.5,1.);
\draw (-2.5028,1.6316) node[anchor=north west] {B};
\draw (1.2172,2.9516) node[anchor=north west] {B};
\draw (1.1772,0.3916) node[anchor=north west] {B};
\draw (2.5172,2.9716) node[anchor=north west] {A};
\draw (2.5972,0.3316) node[anchor=north west] {C};
\draw (7.2972,1.4916) node[anchor=north west] {B};
\draw (9.7372,2.9516) node[anchor=north west] {$D$};
\draw [rotate around={-88.91498238050282:(8.5472,1.2916)},line width=0.8pt] (8.5472,1.2916) ellipse (2.5058836399741433cm and 1.534013063518586cm);
\draw [line width=0.8pt,dash pattern=on 4pt off 4pt] (8.9772,2.8916)-- (8.98,-0.2484);
\draw (8.8372,3.5516) node[anchor=north west] {a};
\draw (8.8972,-0.4884) node[anchor=north west] {c};
\draw (8.8672,1.7316) node[anchor=north west] {$R \in \mathcal{L} '$};
\begin{scriptsize}
\draw[color=black] (-0.7728,2.6516) node {$f_1$};
\draw[color=black] (-0.6928,1.1116) node {$f_2$};
\draw[color=black] (5.3472,2.7316) node {$g_1$};
\draw[color=black] (5.4872,1.0916) node {$g_2$};
\draw [fill=black] (8.9772,2.8916) circle (2.5pt);
\draw [fill=black] (8.98,-0.2484) circle (2.5pt);
\end{scriptsize}
\end{tikzpicture}
\end{center}

We can view a free amalgamation class as a special case of semi-free amalgamation classes. 

In this paper, we will only consider a language $\lan$ consisting of binary, symmetric and irreflexive relations and classes of complete $\lan$-structures. We say an $\lan$-structure $A$ is \emph{complete} if every two distinct elements $a,b \in A$ are related by exactly one relation. We denote this relation by $\rel (a,b)$.

In the appendix of \cite{cherlin1998classification}, Cherlin identified 28 semi-free amalgamation classes of complete structures for languages consisting of three and four relations, specified by triangle constraints, which are defined as the following:

\begin{dfn}\label{forbs}
An $\lan$-structure is a \emph{triangle} if it is a complete structure on three points. Let $S$ be a set of triangles. We define $Forb_c(S)$ to be the set of all complete structures that do not embed any triangle from $S$. We call $S$ \emph{the set of forbidden triangles} of $Forb_c(S)$.
\end{dfn}
We can think of the structures in $Forb_c(S)$ as complete edge-coloured graphs that do not embed some coloured triangles by taking the elements of the structures as vertices and the relations as colours. Throughout this paper, $S$ is assumed to be a set of forbidden triangles such that the corresponding $Forb_c(S)$ is a semi-free amalgamation class. Then we can take its Fra{\"\i}ss{\'e} limit and denote it by $\M_S$.

Our main goal is to prove the simplicity of the automorphism groups of $\M_S$ for $S$ listed in the appendix of  \cite{cherlin1998classification} as well as some general cases. Motivated by Cherlin's examples, we will define a special semi-free amalgamation class, called \emph{prioritised semi-free amalgamation class} (see Definition \ref{priority}). We will prove, in Section 3, that if $Forb_c(S)$ forms a prioritised semi-free amalgamation class, then we can find a stationary independence relation on $\M_S$. We then apply Tent and Ziegler's method in Section 4 and 5 to show the simplicity of $Aut(\M_S)$ for $S$ satisfying some extra conditions (Condition \ref{maincond}).  The main result of this paper is stated in Corollary \ref{maincoro}. In Section 6.1, we will define two conditions on $S$ and show that if $S$ satisfies one of the conditions, $Forb_c(S)$ forms a prioritised semi-free amalgamation class and satisfies Condition \ref{maincond}. We observe that 27 of the 28 examples in Cherlin's list satisfy one of the conditions. These conditions also apply to some general cases where the language consists of more than four relations. Hence, we can apply the main result to Cherlin's examples as well as some general cases. The remaining case, \# 26, needs some extra care and is dealt with in Section 6.2. Combining Section 6.1 and Section 6.2, we prove the following theorem.

\begin{thm}\label{maincherlin}
With the above notation, let $\M_S$ be one of the countable homogeneous structures listed in the appendix of \cite{cherlin1998classification} (see Table 1) and $\M_S$ as defined above. Then $Aut(\M_S)$ is simple.
\end{thm}



\subsection{Background}

In this section, we introduce some concepts in model theory for readers not familar with them.

We first fix a first-order relational language $\lan$, which is specified by a set of relation symbols $\{R_i : i \in I \}$ and each $R_i$ has arity $r_i \in \mathbb{N}$. Then an $\lan$-structure is a set $A$ together with a subset $R^{A}_i \subseteq A^{r_i}$ for each $i \in I$ representing the structure on $A$. In this paper, since the graphs we are working on are all undirected and loopless, the relations are always binary, i.e. $r_i=2$ for all $i$, symmetric and irreflexive. An $\lan$-structure $B$ is a \emph{substructure} of $A$ if $B \subseteq A$ and $R^{A}_i \cap B^{r_i}=R^{B}_i$ for each $i \in I$. For example, if $A$ is a graph, then a substructure of $A$ is an induced subgraph. 

Let $A,B$ be finite substructures of $\M$, we use the notation $AB$ to denote the substructure of $\M$ on the underlying set $A \cup B$. We also simplify the notation $\{a\}B$ to $aB$. Let $G=Aut(\M)$ and denote the pointwise stabiliser of $B$ by $G_{(B)}$ and the setwise stabiliser of $B$ by $G_{\{B\}}$. For a homogeneous structure $\M$, the model theoretic notion of an \emph{$n$-type over $B$} corresponds to a $G_{(B)}$-orbit of an $n$-tuple. For $a \in \M^n$, the \emph{type of $a$ over $B$}, denoted by $tp(a/B)$, is the type over $B$ whose corresponding $G_{(B)}$-orbit contains $a$. 
So, we may use $tp(a/B)$ to denote its corresponding $G_{(B)}$-orbit. Note that $a,a'$ have the same type over $B$ if they lie in the same $G_{(B)}$-orbit, i.e. there exists an automorphism of $\M$ that takes $a$ to $a'$ and fixes $B$ pointwise. We say $a$ \emph{realise} some type $p$ over $B$ if it lies in the corresponding $G_{(B)}$-orbit. We say a type is \emph{algebraic} if its set of realisations is finite, and \emph{non-algebraic} otherwise. In all of our examples, $tp(a/B)$ is algebraic if and only if $a \in B$.

In this paper, the structures we study are all homogeneous and are constructed from amalgamation classes, which are defined as follows:

\begin{dfn}\label{amalgamationclass}
A set of finite $\lan$-structures $\mathcal{C}$ is an \emph{amalgamation class} if it satisfies the following conditions:
\begin{enumerate}[(i)]
\item it is closed under isomorphism and there are countably many isomorphism types.
\item if $A \in \cl$ and $B$ is a substructure of $A$, then $B \in \cl$.
\item (Joint embedding property) if $A,B \in \cl$, then there exist $C \in \cl$ and embeddings $f_1: A \rightarrow C, f_2: B \rightarrow C$
\item (Amalgamation property) if $A,B,C \in \cl$ and $f_1:B \rightarrow A, f_2:B\rightarrow C$ are embeddings, then there exist $D\in \mathcal{C}$ and embeddings $g_1:A \rightarrow D, g_2:C\rightarrow D$ such that $g_1f_1=g_2f_2$
\end{enumerate}
\end{dfn}
 
In the amalgamation property, we can take the embeddings to be inclusion maps. Hence in this paper, when we have an amalgamation class $\cl$ and say that $B$ is a substructure of $A,C$ for $A,B,C \in \cl$, the precise meaning is that there are embeddings from $B$ to $A,C$ as in the amalgamation property. Since $\lan$ is relational, it is enough to check all $A,B,C \in \cl$ such that $|A \setminus B|= |C \setminus B| =1$ in verifying the amalgamation property.

Fra{\"\i}ss{\'e}'s theorem gives a one-to-one correspondence between countable homogeneous structures and amalgamation classes. The theorem states that given an amalgamation class $\cl$, we can construct a countable homogeneous structure $\M$, whose class of isomorphism types of finite substructures is $\cl$. Conversely, the set of all isomorphism types of finite substructures of an homogeneous structure $\M$ forms an amalgamation class, called the \emph{age} of $\M$. By the construction in the proof of Fra{\"\i}ss{\'e}'s theorem, the homogeneous structure satisfies the \emph{Extension Property}: if $A \subseteq \M$ and $f: A \rightarrow B$ is an embedding where $B \in \cl$, then there is an embedding $g: B \rightarrow \M$ such that $g(f(a))=a$ for all $a \in A$.

\subsection{Some Remarks}

We make some remarks about overlaps between this work and \cite{prague2017ramsey}, which studies metrically homogeneous graphs. For any undirected graph, we can put a metric on the graph by defining the distance between any two vertices to be the length of the shortest path between them. A graph is metrically homogeneous if it is homogeneous as a metric space. Cherlin \cite{cherlin2011two} produced a list of such graphs, which is conjectured to be complete. 

In \cite{cherlin2011two}, Cherlin noted that some of the graphs $\M_S$ we study in this paper can be regarded as metrically homogeneous graphs by interpreting the relations as distances. Moreover, some of the graphs can be interpreted as metrically homogeneous graphs in different ways. More precisely, in Table 1 in Section 2, cases $\# 1$ to $\# 8$, $\# 15$ to $\# 18$ and $\# 22$ to $\# 24$ in the list are metrically homogeneous if we let $Y=4, R=3,G=2,X=1$; cases $\# 22, \# 23$ and $\# 25$ produce different metrically homogeneous graphs if we let $X=4, R=3, G=2, Y=1$. This interpretation also works for cases $\# 21$ and $\# 26$. Case $\# 26$ also has interpretation $G=4,R=3,Y=2,X=1$. The full list can be found in \cite{cherlin2011two}. \cite{cherlin2019book} is a more recent work by Cherlin that provides a more thorough study on these graphs.

\cite{prague2017ramsey} showed that all metrically homogeneous graphs with some exceptions in Cherlin's catalogue have the Ramsey property by finding a $\textquoteleft
$completion algorithm', which also provides stationary independence relations on those graphs. This work overlaps with some of the results in this paper. The completion algorithm is somewhat similar to our prioritised semi-free amalgamation process, defined in the next section. The paper \cite{hmn2017generalisedmetricspace} generalised these the results of \cite{prague2017ramsey} to Conant's generalised metric spaces \cite{conant2015generalisedmetricspace}. There is also ongoing work of Evans, Hubi{\v{c}}ka, Kone{\v{c}}n{\'y} and the author of this paper (paper in preparation) to apply Tent and Ziegler's method on these generalised metric spaces. 

The author has also generalised Tent and Ziegler's method so that the symmetry axiom in the stationary independence relation is no longer required (paper in preparation). Hence we can also apply the method to the directed graphs in \cite{cherlin1998classification}, which is also refered to as 2-multi-tournaments in \cite{cherlin2019book}. Note that the proofs in Section 3-5 d not depend on symmetry (except the proof of symmetry of the stationary independence relation in Theorem \ref{sir}).

\section{Prioritised Semi-Free Amalgamation classes}

In the appendix of \cite{cherlin1998classification}, Cherlin produced a list of semi-free amalgamation classes for languages consisting of three and four binary relations, excluding free amalgamation classes and those with a non-trivial equivalence relation on the vertices. We provide Cherlin's list below. Each entry in the list is a set of forbidden triangles $S$ and we study $Forb_c(S)$, the set of all complete structures that do not embed any triangle from $S$, as defined in Definition \ref{forbs}. To avoid any confusion in the notation, we change some of the letters used in Cherlin's original list:

\begin{table}[h]

Language: $\{ R, G, X\}$

$\# 1$ \hspace{1cm} RXX  GGX  XXX

Language: $\{ R, G, X, Y\}$

$\# 1$ \hspace{1cm} RXX  GYX  YXX

$\# 2$ \hspace{1cm} RXX  GYX  YXX  XXX

$\# 3$ \hspace{1cm} RXX  GYX  YXX  YYX

$\# 4$ \hspace{1cm} RXX  GYX  YXX  YYY

$\# 5$ \hspace{1cm} RXX  GYX  YXX  YYX  XXX

$\# 6$ \hspace{1cm} RXX  GYX  YXX  XXX  YYY

$\# 7$ \hspace{1cm} RXX  GYX  YXX  YYX  YYY

$\# 8$ \hspace{1cm} RXX  GYX  YXX  YYX  YYY  XXX

$\# 9$ \hspace{1cm} RXX  GYX  YYX  XXX

$\# 10$ \hspace{0.8cm} RXX  GYX  YYX  XXX  YYY

$\# 11$ \hspace{0.8cm} RXX  GGX  YXX  XXX

$\# 12$ \hspace{0.8cm} RXX  GGX  YXX  XXX  YYX

$\# 13$ \hspace{0.8cm} RXX  GGX  YXX  XXX  YYY

$\# 14$ \hspace{0.8cm} RXX  GGX  YXX  YYX  XXX YYY

$\# 15$ \hspace{0.8cm} RXX  GYX  GGX  YXX  XXX

$\# 16$ \hspace{0.8cm} RXX  GYX  GGX  YXX  XXX  YYX

$\# 17$ \hspace{0.8cm} RXX  GYX  GGX  YXX  XXX  YYY

$\# 18$ \hspace{0.8cm} RXX  GYX  GGX  YXX  YYX  XXX  YYY

$\# 19$ \hspace{0.8cm} RXX  GYX  GGX  YYX  XXX

$\# 20$ \hspace{0.8cm} RXX  GYX  GGX  YYX  XXX  YYY

$\# 21$ \hspace{0.8cm} RXX  RYY  GYX  YYX  XXX

$\# 22$ \hspace{0.8cm} RXX  RYY  GYX  YYX  YXX

$\# 23$ \hspace{0.8cm} RXX  RYY  GYX  YYX  YXX  XXX

$\# 24$ \hspace{0.8cm} RXX  RYY  GYX  YXX  XXX  YYY

$\# 25$ \hspace{0.8cm} RXX  RYY  GYX  YXX  YYX  XXX  YYY
 
$\# 26$ \hspace{0.8cm} RRX  RXX  RYY  GYX GXX  YYX  XXX 

$\# 27$ \hspace{0.8cm} RRY  RRX  GYX GXX  GYY  YYX  YXX  XXX  YYY

  \caption{List of forbidden triangles in the appendix of \cite{cherlin1998classification}}
\end{table}

\hspace{0.5cm}

It can be checked that any amalgamation can be completed by a proper subset $\lan' \subset \lan$. Hence, $Forb_c(S)$ is a semi-free amalgamation class for each $S$ in the list. It can be shown that we may take $\lan'=\{R,G\}$ in all the cases except for \# 26. Note that in \cite{cherlin1998classification}, Cherlin stated that we can find a set of solutions $\lan'$ consisting of two relations for each $Forb_c(S)$. However, this is not possible for \# 26. We will prove that we require $\lan'=\{R,G,Y\}$ for \# 26 in Section 6.

We now look at some examples of these amalgamation classes.
 
\begin{ex}
Taking \# 11 as an example, to check that we may take $\lan '=\{R,G\}$, we want to show that any amalgamation can be completed by either $R$ or $G$, i.e. there does not exist an amalgamation $aB$, $cB$ over some finite set $B$ such that $(a,c)$ cannot be coloured by $R$ or $G$. Suppose we have such an amalgamation. Then there exist $b_1,b_2 \in B$ such that $ab_1c$ forbids $\rel (a,c) = R$ and $ab_2c$ forbids $\rel (a,c) = G$. Since the only forbidden triangles containing $R,G$ are $RXX, GGX$, we may assume without loss of generality that $\rel (a,b_1)=\rel  (b_1,c) = X$ and $\rel (a,b_2) = G, \rel (b_2,c) = X$, as shown in the graph below. However, we cannot find a colour for $b_1b_2$ without creating a forbidden triangle as $b_1ab_2$ forbids $\rel (b_1,b_2) = G$ because of the forbidden triangle $GGX$ and $b_1cb_2$ forbids $\rel (b_1,b_2)\in \{R,Y,X\}$ because of the forbidden triangles $RXX, YXX, XXX$. Therefore, such an amalgamation does not exist. Thus, we may take $\lan '=\{R,G\}$. Similarly, we can prove the same statement for the other cases.

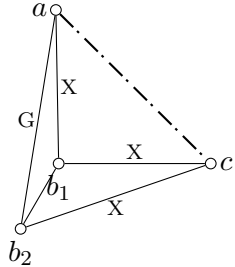
\begin{center}
 \begin{tikzpicture}
    \tikzstyle{every node}=[draw,circle,fill=white,minimum size=4pt,
                            inner sep=0pt]
  \draw (0,0) node (1) [label=left:$a$] {}
     ++(315:2.88cm) node (2) [label=right:$c$]{} 
   ++(180:2cm) node (3) [label=below:$b_1$] {}
     ++ (240:1cm) node (4) [label=below:$b_2$] {};

\draw [thick,dash pattern={on 7pt off 2pt on 1pt off 3pt}] (1) -- (2) ;
\draw (1) -- (3)  node[draw=none,fill=none,font=\scriptsize,midway,right] {X};
\draw (2) -- (3)  node[draw=none,fill=none,font=\scriptsize,midway,above] {X};
\draw (4) -- (3)  node[draw=none,fill=none,font=\scriptsize,midway,right] {};
\draw (1) -- (4)  node[draw=none,fill=none,font=\scriptsize,midway,left] {G};
\draw (2) -- (4)  node[draw=none,fill=none,font=\scriptsize,midway,below] {X};

\end{tikzpicture}
\end{center}
\end{ex}

Also note that given an amalgamation, it is possible to have different completions. For example, if we amalgamate two $R$-coloured edges $ab, bc$ over the vertex $b$ in \# 11, we can colour $(a,c)$ by any relation as there is no forbidden triangle containing $RR$. However, in order to find a stationary independence relation on $\M_S$ that Tent and Ziegler used in \cite{tentziegler2013isometry}, we want to find a $\textquoteleft
$unique' way of amalgamating. In order to do this, we put a linear ordering on $\lan '$.

\begin{dfn}\label{priority}
Let $\lan$ be a language consisting of $n$ binary, symmetric, irreflexive relations. Let $ \lan' \subsetneq \lan$ and suppose $\lan'=\{ R_1,...,R_m\}$. Suppose $\lan'$ is ordered as $R_1 >\cdots >R_m$. For every $A,B,C \in Forb_c(S)$, where $B\subseteq A,C$, define the following way to amalgamate $A$ and $C$ over $B$: for each $a\in A \setminus B, c \in C \setminus B$, first check whether $abc$ form a forbidden triangle for any $b \in B$ if $(a,c) \in R_1$. If $B=\emptyset$ or colouring $(a,c)$ by $R_1$ does not form any forbidden triangle, we let $\rel (a,c) = R_1$. Otherwise, we check the same thing for $(a,c) \in R_2$ and so on so forth. In other word, $\rel (a,c)= R_i$ where $i$ is the smallest possible integer such that $\rel (a,b) \rel (b,c) R_i \notin S$ for any $b \in B$.

Denote the resulting amalgamation by $A \otimes _B C$. If $A\otimes _B C$ does not embed any forbidden triangle, i.e. $A\otimes _B C \in Forb_c(S)$, we call it the \emph{prioritised semi-free amalgamation} of $A, C$ over $B$. If for any $A,B,C \in Forb_c(S)$ where $B \subseteq A,C$, $A \otimes _B C \in Forb_c(S)$, then we say $Forb_c(S)$ is a \emph{prioritised semi-free amalgamation class} with respect to the given ordering on $\lan'$.
\end{dfn}

\begin{ex}

In \# 11, we may let $R>G$. Then $\rel (a,c) \neq R$ in the graph below since otherwise we would have a triangle of $RXX$, so $\rel (a,c) = G$.

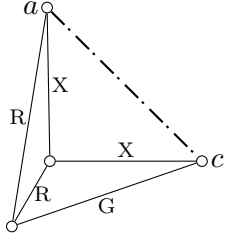
\begin{center}
 \begin{tikzpicture}
    \tikzstyle{every node}=[draw,circle,fill=white,minimum size=4pt,
                            inner sep=0pt]
  \draw (0,0) node (1) [label=left:$a$] {}
     ++(315:2.88cm) node (2) [label=right:$c$]{} 
   ++(180:2cm) node (3) {}
     ++ (240:1cm) node (4) {};

\draw [thick,dash pattern={on 7pt off 2pt on 1pt off 3pt}] (1) -- (2) ;
\draw (1) -- (3)  node[draw=none,fill=none,font=\scriptsize,midway,right] {X};
\draw (2) -- (3)  node[draw=none,fill=none,font=\scriptsize,midway,above] {X};
\draw (4) -- (3)  node[draw=none,fill=none,font=\scriptsize,midway,right] {R};
\draw (1) -- (4)  node[draw=none,fill=none,font=\scriptsize,midway,left] {R};
\draw (2) -- (4)  node[draw=none,fill=none,font=\scriptsize,midway,below] {G};

\end{tikzpicture}
\end{center}

\end{ex}

It is possible that the resulting amalgamation is not in $Forb_c(S)$. For example, in \#11 with order $G>R$, to complete the following amalgamation, we let both $(a_1,c), (a_2,c) \in G$, but we would get a forbidden triangle $GGX$ at $a_1a_2c$.

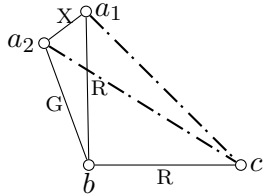
\begin{center}
 \begin{tikzpicture}
    \tikzstyle{every node}=[draw,circle,fill=white,minimum size=4pt,
                            inner sep=0pt]
  \draw (0,0) node (1) [label=right:$a_1$] {}
     ++(315:2.88cm) node (2) [label=right:$c$]{} 
   ++(180:2cm) node (3)  [label=below:$b$] {}
   ++(110:1.72cm) node (4) [label=left:$a_2$] {};

\draw [thick,dash pattern={on 7pt off 2pt on 1pt off 3pt}] (4) -- (2) ;
\draw [thick,dash pattern={on 7pt off 2pt on 1pt off 3pt}] (1) -- (2) ;
\draw (1) -- (3)  node[draw=none,fill=none,font=\scriptsize,midway,right] {R};
\draw (2) -- (3)  node[draw=none,fill=none,font=\scriptsize,midway,below] {R};
\draw (4) -- (3)  node[draw=none,fill=none,font=\scriptsize,midway,left] {G};
\draw (1) -- (4)  node[draw=none,fill=none,font=\scriptsize,midway,above] {X};
\end{tikzpicture}
\end{center}

Note that in $A \otimes _B C$, a forbidden triangle can only appear in $AC \setminus B$ as we ensure by construction that there is no forbidden triangle with a vertex from $B$.

\section{Stationary Independence Relations}

Following \cite{tentziegler2013isometry}, we consider a ternary relation among finite substructures of a homogeneous structure, called a \emph{stationary independence relation}. We will show in this section that the Fra{\"\i}ss{\'e} limit of a prioritised semi-free amalgamation class has a stationary independence relation.

\begin{dfn}\label{sirdfn}

Let $\M$ be a homogeneous structure and suppose $A\ind _B C$ is a ternary relation among finite substructure $A,B,C$ of $\M$. We say that $\ind$ is a \emph{stationary independence relation} if the following axioms are statisfied:

\begin{enumerate}[(i)]
\item Invariance: for any $g \in Aut(\M)$, if $A \ind _B C$, then $ gA \ind _{gB} gC$
\item Monotonicity: $A \ind _B CD \Rightarrow A \ind _B C$, $A \ind _{BC} D$
\item Transitivity: $A \ind _B C$, $A \ind _{BC} D \Rightarrow A \ind _B D $
\item Symmetry: $A \ind _B C \Rightarrow C \ind _B A$  
\item Existence: If $p$ is an $n$-type over $B$ and $C$ is a finite set, then $p$ has a realisation $a$ such that $a \ind _B C$.
\item Stationarity: If $a$ and $a'$ are $n$-tuples that have the same type over $B$ and are both independent from $C$ over $B$, then $a$ and $a'$ have the same type over $BC$.
\end{enumerate}
We say \emph{$A$ is independent from $C$ over $B$} if $A \ind _B C$.
\end{dfn}

\begin{rmk}
It can be shown from the above axioms that 
\[A \ind _B C \Leftrightarrow AB \ind _B C \Leftrightarrow A\ind_B BC\]
Hence, we can assume without loss of generality that $B\subseteq A,C$ whenever $A\ind _B C$ for arbitrary $A,B,C$.
\end{rmk}


\begin{thm}\label{sir}
Let $S$ be a set of forbidden triangles such that $Forb_c(S)$ is a prioritised semi-free amalgamation class. Let $\M_S$ be the Fra{\"\i}ss{\'e} limit of $Forb_c(S)$. For any finite substructure $ A,B, C$ of $\M_S$, let $A\ind _B C$ if $ABC=AB \otimes _B BC$ where $AB \otimes _B BC$ is the prioritised semi-free amalgamation defined in Definition \ref{priority}. Then $\ind$ is a stationary independence relation on $\M_S$.
\end{thm}

\begin{proof}
\underline{Invariance:} Given $B \subseteq A,C \subseteq \M_S$, the prioritised semi-free amalgamation of $A,C$ over $B$ is dependent only on the types of $A,B,C$ and homogeneity of $\M_S$.

\underline{Monotonicity:} Let $B \subseteq A \subseteq \M_S$ and $ B\subseteq C \subseteq D \subseteq \M_S$. Suppose $A \ind _B D$, i.e. $AD=A\otimes _B D$, we want to show $A\ind _B C$ and $AC \ind _C D$. By construction, we have $A \otimes _B C \subseteq A \otimes _B D=AD$, so $A \otimes _B C=ABC$, i.e. $A\ind _B C$.

 
To show that $AC \ind_C D$, for any $a \in A \setminus C,d\in D \setminus C$, we want to show that $(a,d)$ is completed by the same relation in both $A \otimes _B D$ and $AC\otimes _C D$. Suppose $(a,d)$ is completed by some $R_i \in \lan'$ in $A \otimes _B D$ and $(a,d) \notin R_i$ in $AC\otimes _C D$. Then there exists $c \in C\setminus B$ that forbids $(a,d)$ to be coloured b $R_i$, i.e. $R_i \rel (a,c) \rel (c,d)$ forms a forbidden triangle. However, $ (a,c), (c,d)$ have the same colours in $A \otimes _B D$ as in $A \otimes _C D$, i.e. it would also be a forbidden triangle in $A \otimes _B D$, a contradiction.

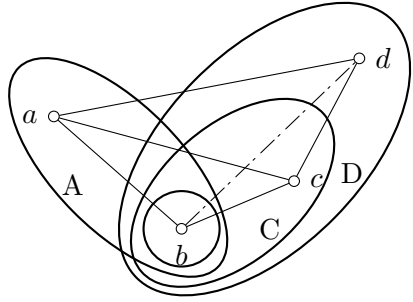
\begin{tikzpicture}[line cap=round,line join=round,>=triangle 45,x=1cm,y=1cm]
\clip(-6.,-1.) rectangle (7.,4.);
  \draw (0,0) node (1) [draw,circle,fill=white,minimum size=4pt,
                            inner sep=0pt,label=below:$b$] {};
  \draw (1.48,0.632) node (2) [draw,circle,fill=white,minimum size=4pt,
                            inner sep=0pt,label=right:$c$]{};
  \draw (2.338,2.26) node (3) [draw,circle,fill=white,minimum size=4pt,
                            inner sep=0pt,label=right:$d$]{}; 
  \draw (-1.68,1.49) node (4) [draw,circle,fill=white,minimum size=4pt,
                            inner sep=0pt,label=left:$a$]{}; 
\draw (1) -- (2)  node[draw=none,fill=none,font=\scriptsize,midway,below] {};
\draw (2) -- (3)  node[draw=none,fill=none,font=\scriptsize,midway,right] {};
\draw [dash pattern={on 7pt off 2pt on 1pt off 3pt}] (1) -- (3) ;
\draw (1) -- (4)  node[draw=none,fill=none,font=\scriptsize,midway,left] {};
\draw (4) -- (2)  node[draw=none,fill=none,font=\scriptsize,midway,right] {};
\draw (4) -- (3)  node[draw=none,fill=none,font=\scriptsize,midway,above] {};
\draw [line width=0.8pt] (0.,0.) circle (0.5cm);
\draw [rotate around={-45.47753180927141:(-0.8333012851325438,0.8136896716187254)},line width=0.8pt] (-0.8333012851325438,0.8136896716187254) ellipse (1.8608157008316764cm and 0.83313683507826cm);
\draw [rotate around={41.10943160487273:(0.671073186784725,0.47701744109573757)},line width=0.8pt] (0.671073186784725,0.47701744109573757) ellipse (1.5907259683603852cm and 0.898717957730353cm);
\draw [rotate around={45.68415840134765:(1.091073186784725,1.0570174410957374)},line width=0.8pt] (1.091073186784725,1.0570174410957374) ellipse (2.3968951528675992cm and 1.295311909750448cm);
\draw (-1.71,0.852) node[anchor=north west] {A};
\draw (0.886,0.26) node[anchor=north west] {C};
\draw (1.958,1.0) node[anchor=north west] {D};
\end{tikzpicture}

\underline{Transitivity:} Let $B \subseteq A \subseteq \M_S, B\subseteq C \subseteq D \subseteq \M_S$. Suppose $A\ind _B C$ and $A\ind _C D$, i.e. $AC=A \otimes _B C$, $AD=AC \otimes _C D$. We want to show $A \ind _B D$, i.e. $A\otimes _B D =AD=AC \otimes _C D$. We already have $A \otimes _B C =AC \subset AD =AC \otimes _C D$, so we only need to show that, for any $a \in A \setminus B, d \in D \setminus B$, $(a,d)$ is coloured by the same relation in both $A \otimes _B D$ and $AC \otimes _C D$. Suppose $(a,d)$ is completed by some $R_{i}\in \lan' $ in $A \otimes_B D$ and $(a,d) \notin R_i$ in $AC \otimes _C D$. This implies that there exists $c\in C \setminus B$ that forbids $(a,d)$ to be coloured by $R_i$, i.e. $R_{i} \rel (a,c) \rel (c,d)$ is a forbidden triangle. This would be a forbidden triangle in $A\otimes_B D$, a contradiction.

\underline{Symmetry:} The symmetry of the independence relation follows from the symmetry of the relations in the language.

\underline{Existence:}  Let $p$ be a type over $B$ and $C$ any finite set. Let $a$ be a realisation of $p$, we can embed $BC$ into $aB \otimes _B BC$. Then by the extention property, we may assume $aB \otimes _B BC \subseteq \M_S$ and $a \ind _B C$.

\underline{Stationarity:} Suppose $a,a'$ have the same type over $B$ and are both independent from $C$ over $B$. Then $aB \otimes _B BC$ is isomorphic to $a'B \otimes_B BC$ since the relations between $a,a'$ and any $ c\in C$ depend only on the relations between $aB, a'B$ and $cB$ respectively. Hence, $aBC=aB \otimes _B BC$ is isomorphic to $a'B \otimes_B C=a'BC$, which implies $a,a'$ have the same type over $BC$.
\end{proof}

Therefore, we have shown that if $Forb_c(S)$ forms a prioritised semi-free amalgamation class, then there is a stationary independence relation on its Fra{\"\i}ss{\'e} limit $\M_S$.

\section{Preparatory Result}

For Section 4 and Section 5, we will only consider $S$ such that $Forb_c(S)$ forms a prioritised semi-free amalgamation class as defined in Definition \ref{priority}. $\M$ in this section is always a relational homogeneous structure.

Theorem \ref{tztheorem} was proved in \cite{tentziegler2013isometry}, which uses the result from \cite{macphersontent2011simplicity} that a non-trivial automorphism of a free homogeneous structure does not fix the set of realisations of any non-algebraic type pointwise. In this section, we will prove the same statement for $\M_S$, where $S$ satisfies the following condition.

\begin{con}\label{maincond}
Let $S$ be a set of forbidden triangles. Assume that
\begin{enumerate}[(i)]
\item $S$ does not contain any triangle involving $R_1R_1$ or $R_1 R_2$.
\item Let $a,b, c \in \M_S$ and $B \subseteq \M_S$ such that $a \ind_{bB} c$. If $\rel (a,b) \in \lan'$, we have $a \ind _B c$.
\end{enumerate}
\end{con}
Note that satisfying Condition \ref{maincond} does not guarantee that $Forb_c(S)$ forms a prioritised semi-free amalgamation class.

For readers familiar with model theory, the idea in this section comes from \emph{imaginary elements}. Since the proof here does not directly involve them, I refer interested readers to Section 16.4 and 16.5 in \cite{poizat2012course}. We will show that our structure $\M_S$ has the \emph{intersection property}, defined in the following. Note that for our structure $\M_S$, the intersection property is equivalent to having \emph{weak elimination of imaginaries} by Theorem 16.17 in \cite{poizat2012course} since the algebraic closure is trivial in $\M_S$. We will use this property of $\M_S$ to prove that any non-trivial automorphism of $\M_S$ does not fix the set of realisations of any non-algebraic type pointwise.

\begin{dfn}
Let $\M$ be a homogeneous structure and $G=Aut(\M)$. We say that $\M$ has the \emph{intersection property}, if for all finite subsets $A,B \subseteq \M$,
\[ \langle G_{(A)}, G_{(B)} \rangle =G _{(A \cap B)}. \]
\end{dfn}

Note that if $g \in \langle G_{(A)}, G_{(B)} \rangle$, then $g$ fixes $A\cap B$ pointwise. So, we always have $\langle G_{(A)}, G_{(B)}\rangle$$\leq G _{(A \cap B)}$. 

\begin{prop}\label{intersectiononlyone}
Let $\M$ be a homogeneous structure and $G=Aut(\M)$. Suppose 
\begin{align}\label{homointersection}
\langle G_{(A)}, G_{(B)} \rangle =G _{(A \cap B)}
\end{align}
for all finite $A,B \subseteq \M$ such that $|A \setminus B|=|B \setminus A|=1$. Then (\ref{homointersection}) holds for all finite subsets $A,B \subseteq \M$. In other words, to prove the intersection property for a relational homogeneous structure $\M$, it is enough to show (\ref{homointersection}) for finite $A,B \subseteq \M$ such that $|A \setminus B|=|B \setminus A|=1$. 
\end{prop}

\begin{proof}
We prove this by induction on $n=|A \setminus B|+|B \setminus A|$. The induction base is given by the assumption.

Let $n \geq 3$ and assume $|A \setminus B| \geq 2$. Pick $a \in A \setminus B$ and let $A'=(A \cap B) \cup \{ a \}$ and $B'=B \cup \{a \}$. Then,
\[\langle G_{(A)}, G_{(B)} \rangle \geq \langle G_{(A)}, G_{(B')},G_{(B)} \rangle .
 \]
By inductive hypothesis, we have
\[ \langle G_{(A)},G_{(B')} \rangle =G_{(A')} .
\]
Hence,
\[\langle G_{(A)}, G_{(B)} \rangle \geq \langle G_{(A')},G_{(B)} \rangle . \]
Again, by inductive hypothesis,
\[\langle G_{(A)}, G_{(B)} \rangle \geq \langle G_{(A')},G_{(B)} \rangle =G_{(A \cap B)}  .\]

Since $\langle G_{(A)}, G_{(B)} \rangle \leq G _{(A \cap B)}$, we have $\langle G_{(A)}, G_{(B)} \rangle =G _{(A \cap B)}$.
\end{proof}

\begin{lem}\label{intersectionlemma}
Let $\M$ be a homogeneous structure with a stationary independence relation $\ind$ and $G=Aut(\M)$. If $A \ind _{A \cap B} B$, then $\langle G_{(A)}, G_{(B)} \rangle =G _{(A \cap B)}$.
\end{lem}

\begin{proof}

Let $g \in G _{(A \cap B)}$. By the existence axiom of $\ind$, there exists $A'$ having the same type as $A$ over $B$ such that 
\[ A' \ind _{B} B \cup gB .
\]

Since $\M$ is homogeneous and $tp(A/B)=tp(A'/B)$, there is $k \in G_{(B)}$ such that $kA =A'$. Hence $G_{(A')} =kG_{(A)} k^{-1}$ and 
\[\langle G_{(A)}, G_{(B)} \rangle \geq \langle G_{(A')},G_{(B)} \rangle .
\]

By invariance, we have $kA \ind _{k(A\cap B)} kB$, i.e.
\[ A' \ind _{A \cap B} B .\]

By transitivity on the above two independence relations, we have 
\[ A' \ind _{A \cap B} B \cup gB .
\]

By monotonicity, we have $ A' \ind _{A \cap B} B $ and $A' \ind _{A \cap B} gB$. We also have $tp(B/A\cap B)=tp(gB/ A \cap B)$. So, by stationarity, $tp (B/ A')= tp(gB /A')$, i.e. there is $h \in G_{(A')}$ such that $h B=gB$. Then $g \in hG_{(B)}$ and hence, $g \in  \langle G_{(A')},G_{(B)} \rangle \leq \langle G_{(A)}, G_{(B)} \rangle$. Therefore, $G _{(A \cap B)} \leq \langle G_{(A)}, G_{(B)} \rangle $.

Since $\langle G_{(A)}, G_{(B)} \rangle \leq G _{(A \cap B)}$, we have the required result.
\end{proof}

Now let $S$ be a set of forbidden triangles satisfying Condition \ref{maincond} and $Forb_c(S)$ be the set of all finite $\lan$-structures that do not embed any triangle from $S$ as defined in Definition \ref{forbs}. Let $\M_S$ be the Fra{\"\i}ss{\'e} limit of $Forb_c(S)$ and $G=Aut(\M_S)$. For any $ A,B, C \subseteq \M_S$, let $A\ind _B C$ if $ABC=AB \otimes _B BC$ where $AB \otimes _B BC$ is the prioritised semi-free amalgamation defined in Definition \ref{priority}. The main result of this section is:

\begin{thm}\label{setequal}
Let $S$ be a set of forbidden triangles satisfying Condition \ref{maincond}. For any $a,c \in \M_S$ and finite subset $B \subseteq \M_S$, 
\begin{align}\label{intersectionequal}
\langle G_{(aB)}, G_{(cB)} \rangle =G _{(B)} .
\end{align}
Hence, by Proposition \ref{intersectiononlyone}, $\M_S$ has the intersection property, and thus weak elimination of imaginaries.
\end{thm}

\begin{proof}
By existence, we can find $c'$ realising $tp(c/aB)$ such that $ c' \ind _{aB} c$. Then there exists $k \in G_{(aB)}$ such that $kc=c'$ and hence, $G_{(c'B)}=kG_{(cB)}k^{-1}$. So we have 
\[ \langle G_{(aB)}, G_{(cB)} \rangle \geq \langle G_{(c'B)}, G_{(cB)} \rangle .
\]

We can also find $c''$ realising $tp(c/c'B)$ such that $c'' \ind _{c'B} c$. Then there exists $h \in G_{(c'B)}$ such that $hc=c''$ and hence, $G_{(c''B)}=hG_{(cB)}h^{-1}$. So we have 
\[ \langle G_{(c'B)}, G_{(cB)} \rangle \geq \langle G_{(c''B)}, G_{(cB)} \rangle .
\]

Since $ c' \ind _{aB} c$, we have $\rel (c',c) \in \lan'$. By part (ii) of Condition \ref{maincond}, we can obtain  $c'' \ind _{B} c$ from  $c'' \ind _{c'B} c$.

Then by Lemma \ref{intersectionlemma}, $ \langle G_{(c''B)}, G_{(cB)} \rangle = G_{(B)}$. Thus,
\[  \langle G_{(aB)}, G_{(cB)} \rangle \geq G_{(B)} .
\]
Since $ \langle G_{(aB)}, G_{(cB)} \rangle \leq  G_{(B)}$, we have $\langle G_{(aB)}, G_{(cB)} \rangle =G_{(B)}$.

\end{proof}

We now prove that if $g \in Aut(\M_S)$ fixes some non-algebraic type over $B$ setwise, then it fixes $B$ setwise. We say $P$ is a \emph{definable set over $B$} if it is a union of $G_{(B)}$-orbits and $B$ is a \emph{defining set for $P$}. Note that for $P \subseteq \M_S$ if $B_1,B_2$ are defining sets for $P$, then $\langle G_{(B_1)}, G_{(B_2)} \rangle \leq G_{\{P\}}$. So, by the previous theorem, $B_1 \cap B_2$ is a defining set for $P$ and hence it follows that $P$ has a minimal defining set.

\begin{coro}\label{fixsetwise}
Let $P \subseteq \M_S$ be an infinite definable set and $B$ is a minimal defining set for $P$. If $g\in Aut(\M_S)$ fixes $P$ setwise, then $g$ fixes $B$ setwise.
\end{coro}

\begin{proof}
Since $g$ fixes $P$ setwise, $gB$ is also a defining set for $P$. Hence, by minimality of $B$, we have $gB=B$.
\end{proof}

Since the set of realisations of a non-algebraic type is also a definable set, we have that if $g \in Aut(\M_S)$ fixes some non-algebraic type over $B$ setwise, then it fixes $B$ setwise.

\begin{prop}\label{nofix}
Let $g \in Aut(\M_S)$. For any non-algebraic type $p$ over some finite set $B$, if $g$ fixes its set of realisations pointwise, then $g=1$.
\end{prop}
 
\begin{proof}
Without loss of generality, we may assume $B$ is a minimal defining set for $p$. Let $c \in \M_S$ be such that $c \ind B$. Then $\rel (c,b) =R_1$ for every $b \in B$.

\underline{Claim 1}: if $g$ fixes $p$ pointwise, then $g$ fixes $tp(c /B )$ pointwise.

For $c_1,c_2 \in tp(c /B )$ distinct, there exists $a$ realising $p$ such that $\rel (c_1 ,a )=R_1, \rel(c_2,a)=R_2$ since $\rel(c_1 ,b)=\rel (c_2 ,b)=R_1$ for every $b \in B $ and there does not exist forbidden triangle containing $R_1R_1$ or $R_1R_2$ as required by Condition \ref{maincond}. Since $ga=a$, $gc_1  \neq c_2$.

Since $tp(c/B)$ contains every point in $\M_S$ that is in relation $R_1$ from $B$, $g$ fixes $tp(c/B )$ setwise. Therefore, $gc_1=c_1$, i.e. $g$ fixes $tp(c /B )$ pointwise.

\underline{Claim 2}: if $g$ fixes $tp(c /B)$ pointwise, then $g=1$.

For $x,y \notin B$ distinct, there exists $c'$ realising $tp(c/B  )$ such that $\rel (x,c')=R_1, \rel(y,c')=R_2$ since $\rel(b,c')=R_1$ for every $b \in B$ and there does not exist forbidden triangle containing $R_1R_1$ or $R_1R_2$. As $gc'=c'$, $gx \neq y$. By Lemma \ref{fixsetwise}, $g$ fixes $B $ setwise, $gx \notin B $. Therefore, $gx=x$ for every $x \notin  B$.

Then for every $b,b'\in B  $ distinct, there exists $d \notin B  $ such that $\rel(d,b)=R_1, \rel (d,b')=R_2$. Hence $gb=b$ for every $b \in B$.

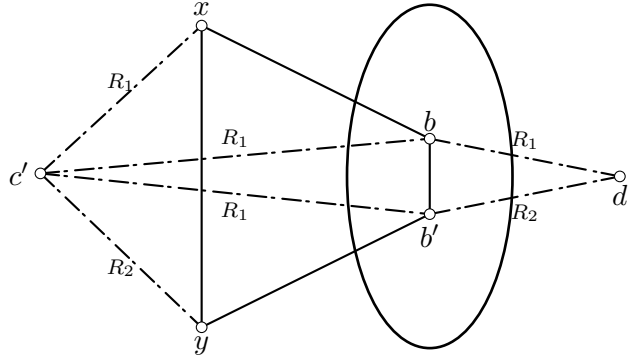
\begin{tikzpicture}[line cap=round,line join=round,>=triangle 45,x=1cm,y=1cm]
\clip(-4.12,-3.1) rectangle (5.12,1.94);
\draw [rotate around={90:(2.5,-0.5)},line width=1pt] (2.5,-0.5) ellipse (2.276829194552678cm and 1.0880952077678654cm);
\tikzstyle{every node}=[draw,circle,fill=white,minimum size=4pt,
                            inner sep=0pt]
\draw [line width=0.8pt] (-0.5,1.5)-- (2.5,0) ;
\draw [line width=0.8pt] (2.5,0)-- (2.5,-1);
\draw  [line width=0.8pt] (2.5,-1)-- (-0.5,-2.5);
\draw  [thick,dash pattern={on 7pt off 2pt on 1pt off 3pt}] (-2.62,-0.46)-- (2.5,0)  node[draw=none,fill=none,font=\scriptsize,midway,above] {$R_1$};
\draw  [thick,dash pattern={on 7pt off 2pt on 1pt off 3pt}] (-2.62,-0.46)-- (2.5,-1)  node[draw=none,fill=none,font=\scriptsize,midway,below] {$R_1$};
\draw  [thick,dash pattern={on 7pt off 2pt on 1pt off 3pt}] (-2.62,-0.46)-- (-0.5,1.5)  node[draw=none,fill=none,font=\scriptsize,midway,above] {$R_1$};
\draw  [thick,dash pattern={on 7pt off 2pt on 1pt off 3pt}] (-2.62,-0.46)-- (-0.5,-2.5)  node[draw=none,fill=none,font=\scriptsize,midway,below]{$R_2$};
\draw  [thick,dash pattern={on 7pt off 2pt on 1pt off 3pt}] (5,-0.5)-- (2.5,0)  node[draw=none,fill=none,font=\scriptsize,midway,above] {$R_1$};
\draw  [thick,dash pattern={on 7pt off 2pt on 1pt off 3pt}] (5,-0.5)-- (2.5,-1)  node[draw=none,fill=none,font=\scriptsize,midway,below] {$R_2$};
\draw  [line width=0.8pt] (-0.5,1.5)-- (-0.5,-2.5);
\draw  (-0.5,1.5) node (1) [draw,circle,fill=white,minimum size=4pt,
                            inner sep=0pt,label=above:$x$]{};;
\draw  (2.5,0)  node (2) [draw,circle,fill=white,minimum size=4pt,
                            inner sep=0pt,label=above:$b$]{};
\draw  (2.5,-1)  node (3) [draw,circle,fill=white,minimum size=4pt,
                            inner sep=0pt,label=below:$b' $]{};                    
\draw (-0.5,-2.5)  node (4) [draw,circle,fill=white,minimum size=4pt,
                            inner sep=0pt,label=below:$y$]{};
\draw  (-2.62,-0.46)  node (5) [draw,circle,fill=white,minimum size=4pt, inner sep=0pt,label=left:$c'$]{};
\draw (5,-0.5)  node (6) [draw,circle,fill=white,minimum size=4pt, inner sep=0pt,label=below:$d$]{};
\end{tikzpicture}

Therefore, combining the above two claims, we have the required result.

\end{proof}

We have shown that for any non-trivial $g \in Aut(\M_S)$ and any non-algebraic type $p  $ over some finite set $B$, $g$ moves some realisation $a  $ of $p  $. We can then define a new type $q  $ over $a   B$ which is defined the same over $B$ as $p  $. This is possible because of the amalgamation property. Since we can repeat this process countably many times, $g$ moves countably many realisations of $p$.

\section{Simplicity of the Automorphism Groups}

In this section, with the help of automorphisms that \emph{move almost maximally}, defined below, we apply the following theorem from \cite{tentziegler2013isometry}, which Tent and Ziegler used to study the automorphism group of the Urysohn Space. 

\begin{dfn}
Let $\M$ be a homogeneous structure with a stationary independence relation $\ind$. We say that $g \in Aut(\M)$ \emph{moves almost maximally} if every 1-type over a finite set $X$ has a realisation $a$ such that $a \ind _X ga$.
\end{dfn}

\begin{thm}\label{tzmm}
(\cite{tentziegler2013isometry}, 5.4) Suppose that $\M$ is a countable homogeneous structure with a stationary independence relation. If $g\in Aut(\M)$ moves almost maximally, then any element of $Aut(\M)$ is the product of sixteen conjugates of $g$.
\end{thm}

In this section, we prove the following theorem for $\M_S$ where $S$ satisfies Condition \ref{maincond}. This then implies the simplicity of $Aut(\M_S)$ by Theorem \ref{tzmm}. We let $[h,g]$ denote the \emph{commutator} $h^{-1} g^{-1}hg$.

\begin{thm}\label{main}
Let $S$ be a set of forbidden triangles satisfying Condition \ref{maincond} such that $Forb_c(S)$ forms a prioritised semi-free amalgamation class. Let $\M_S$ be the Fra{\"\i}ss{\'e} limit of $Forb_c(S)$. Then for any non-trivial automorphism $g\in Aut(\M_S)$, there exist $k,h \in Aut(\M_S)$ such that $[k,[h,g]]$ moves almost maximally.
\end{thm}

This immediately implies the following:
\begin{coro}\label{maincoro}
Under the same assumptions as in the previous theorem, $Aut(\M_S)$ is simple. In particular, if $1 \neq g \in Aut(\M_S)$, then every element of $Aut(\M_S)$ can be written as a product of 64 conjugates of $g^{\pm 1}$. 
\end{coro}

In order to prove Theorem \ref{main}, we show the following results.

\begin{lem}\label{colourrange}
Let $S$ be a set of forbidden triangles such that $Forb_c(S)$ is a semi-free amalgamation class. Let $\M_S$ be its Fra{\"\i}ss{\'e} limit. Then given any non-trivial $g \in Aut(\M_S)$, we can construct $h\in Aut(\M_S)$ such that for any non-algebraic 1-type $p$ over some finite set $X$, there exist infinitely many realisations $a$ of $p$ such that $\rel (a, [h,g]a) \in \lan '$.
\end{lem}
\begin{proof}
List all 1-types over a finite set as $p_1,p_2,...$. We start with the empty map and use a back-and-forth construction to build $h$. Suppose at the some stage, we have a partial isomorphism $\tilde{h}:A \rightarrow B$ such that for any $p_j \in \{p_1,...,p_{i-1}\}$, there exists a realisation $a$ of $p_j$ such that $\rel (a, [\tilde{h},g]a) \in \lan '$.

Let $p:=p_i$ be a 1-type over $X$. We want to extend $\tilde{h}$ such that $p$ has a realisation $a$ such that $\rel (a, [\tilde{h},g]a)  \in \lan '$. We may assume $X \subseteq A$ by extending $\tilde{h}$.

Since $p$ is non-algebraic, by Proposition \ref{nofix}, $p$ has a realisation $a$ such that $a \notin A \cup g^{-1} (aA)$ and $\tilde{h} \cdot tp(a/A)$ has a realisation $b$ such that $b\notin B \cup g^{-1} (bB)$. Extend $\tilde{h}$ by sending $a$ to $b$. It is well-defined since $\tilde{h}\cdot tp(a/A)=tp(b/B)$.

By the extension property, there exists a realisation $c$ of $\tilde{h}^{-1} \cdot tp(gb/bB)$ such that $c$ and $ga$ are semi-freely amalgamated over $aA$. Since $b \notin g^{-1}(bB)$ and $\tilde{h} \cdot tp(c/aA)=tp(gb/bB)$, we have $c\notin aA$. We also have $ga \notin aA$, hence $(c,ga)$ is coloured using relations from $\lan'$, i.e. $\rel (c,ga) \in \lan'$. 

Extend $\tilde{h}$ by sending $c$ to $gb$. Since $\tilde{h} \cdot tp(c/aA)=tp(gb/bB)$, $\tilde{h}$ is a well-defined partal isomorphism. Then $c=\tilde{h}^{-1} g \tilde{h} a$ and we have  
\[ \rel (a,[\tilde{h},g]a) = \rel ( a, \tilde{h}^{-1} g ^{-1} \tilde{h}g a )= \rel (\tilde{h}^{-1} g \tilde{h} a, g a )= \rel (c,ga) \in \lan'. \] 

At every other step, we can make sure $X \subset B$ by extending $\tilde{h}$. Let $h$ be the union of all $\tilde{h}$ over each step, it is an automorphism since it is well defined and bijective as we made sure every finite subset of $\M$ is contained in both domain and image. 
\end{proof}

\begin{thm}\label{movealmost}
Let $S$ be a set of forbidden triangles satisfying Condition \ref{maincond} and $Forb_c(S)$ be a prioritised semi-free amalgamation class as defined in Definition \ref{priority}. Let $\M_S$ be its Fra{\"\i}ss{\'e} limit and $g \in Aut(\M_S)$ be an automorphism of $\M_S$ such that for any non-algebraic 1-type $p$ over some finite set, there exist infinitely many realisations $a$ of $p$ with $\rel (a, g a) \in \lan '$. Then there exists $k \in Aut(\M_S)$ such that $[k,g]$ moves almost maximally. 
\end{thm}
\begin{proof}
We again use a back-and-forth construction as in the previous proof. Suppose at some stage, we have a partial isomorphism $\tilde{k}:A \rightarrow B$. It suffices to show that we can extend $\tilde{k}$ so that $[\tilde{k},g]$ moves some given 1-type $p$ almost maximally. We may assume $p$ is non-algebraic.

We may also assume $X \subseteq A$ and $\tilde{k} X \subseteq g^{-1}B$ by extending $\tilde{k}$, then we have $\tilde{k}^{-1}g\tilde{k} X \subseteq A$.

By the existence axiom of $\ind$ and the assumption about $g$, there exists a realisation $a$ of $p$ such that $\rel (a, ga)   \in \lan '$ and 
\begin{align}\label{aindA}
a \ind _X A.
\end{align}

By existence and Lemma \ref{nofix}, there exists a realisation $b$ of $\tilde{k} \cdot tp(a/A)$ such that $g b \neq b$ and 
\begin{align}\label{bindinv}
b \ind _B g^{-1}B.
\end{align}
Extend $\tilde{k}$ by sending $a$ to $b$. It is a well-defined partial isomorphism since $\tilde{k}\cdot tp(a/A)=tp(b/B)$. From (\ref{aindA}), by invariance, we get 
\begin{align}\label{bindB}
b \ind _{\tilde{k} X} B.
\end{align}

Applying transitivity on (\ref{bindinv}) and (\ref{bindB}), we have 
\begin{align}\label{bindginvb}
b \ind _{\tilde{k} X}g^{-1} B.
\end{align}

Again by the existence axiom, there exists a realisation $c$ of $ \tilde{k}^{-1} \cdot  tp(gb/bB)$ such that 
\begin{align*}
c \ind _{aA} ga.
\end{align*}
Since $\rel (a,ga) \in \lan'$, by part (ii) of Condition \ref{maincond}, we have 
\begin{align}\label{cindga}
c \ind _A ga.
\end{align}

Extend $\tilde{k}$ by sending $c$ to $gb$. We have $tp (gb/bB)=\tilde{k} \cdot tp(c/aA)$ and $gb \neq b$, so $\tilde{k}$ is a well-defined partial isomorphism.

Acting by $\tilde{k}^{-1}g$ on (\ref{bindginvb}), we then get
\begin{align*}
\tilde{k}^{-1}gb \ind _{\tilde{k}^{-1}g\tilde{k} X} \tilde{k}^{-1}B,
\end{align*}
which can be simplified to
\begin{align}\label{cindA}
 c \ind _{\tilde{k}^{-1}g \tilde{k} X} A.
\end{align}

Since $ \tilde{k}^{-1}g \tilde{k} X \subseteq A $, we can apply transitivity on (\ref{cindga}) and (\ref{cindA}) to obtain $ c \ind _{\tilde{k}^{-1}g\tilde{k} X} ga$. Acting by $\tilde{k}^{-1}g^{-1}  \tilde{k}$ on it, we have the required result, i.e.
\begin{align*}
a \ind _X [\tilde{k},g] a.
\end{align*} 

\end{proof}

\begin{proof}[Proof of Theorem \ref{main}]
Given any automorphism $g \in Aut(\M_S)$, we can first find $h \in Aut(\M_S)$ such that for any non-algebraic 1-type $p$ over some finite set, there exist infinitely many realisations $a$ of $p$ such that $\rel (a, [h,g] a) \in \lan '$. Then by the previous theorem, there exists $k \in Aut(\M_S)$ such that $[k,[h,g]]$ moves almost maximally.
\end{proof}

\section{Examples}

\subsection{General Cases}

With notation and assumption as stated in Definition \ref{priority}, we fix $\lan'$ and an ordering on $\lan'$. Let $S$ be a set of forbidden triangles. To make sure $A \otimes _B C \in Forb_c(S)$, we can impose one of the following two conditions on $S$:

\begin{con}\label{condition1}
Assume $S$  does not contain any triangle of the form $R_i R_j R'$ where $R_i, R_j \in \lan '$ and $R' \in \lan$.
\end{con}

\begin{con}\label{condition2}
Assume $\lan'=\{R_1,R_2\}$ with order $R_1 >R_2$ and for some subset $\lan ^{\star} \subseteq \lan \backslash \lan '$ and $\hat{\lan}  \subseteq \lan \backslash (\lan ' \cup \lan^{\star})$, $S$ contains all triangles of the form $R'R^1R^2, R_2R_2R^1, R_2 \hat{R} R^1$, where $R' \in \lan \setminus (\hat{\lan} \cup \{R_2\}),  R^1,R^2 \in \lan ^{\star}, \hat{R} \in \hat{\lan}$ and $S$ contains no other triangle involving $R_1$ or $R_2 R_2$.

\end{con}

In this section, we will show that if $S$ satisfies either one of the above two conditions, then $Forb_c(S)$ forms a prioritised semi-free amalgamation class and satisfies Condition \ref{maincond}. Therefore, by Corollary \ref{maincoro}, the automorphism group of its Fra{\"\i}ss{\'e} limit $\M_S$ is simple. The following lemma states that 27 of 28 Cherlin's examples satsify either one of the conditions can be checked by quick computation. Hence, the automorphism groups of their Fra{\"\i}ss{\'e} limits are also simple.

\begin{lem}
In Cherlin's list, the only case where $\lan$ consists of three relations satisfies Condition \ref{condition2} with $\lan'=\{ R,G\}$, $\lan ^{\star} =\{X\}$ and order $R>G$. For $\lan$ consisting of four relations, $\lan'=\{ R,G\}$ works for all 27 cases execpt case \# 26. Cases \# 1 to \# 10 and cases \#21 to \# 25 satisfy Condition \ref{condition1}. Cases \#11 to \# 18 satisfy Condition \ref{condition2} with $\lan ^{\star} =\{X\}, \hat{\lan}=\emptyset$ and the order $R>G$. Case \# 19 and case \# 20 satisfy Condition \ref{condition2} with $\lan ^{\star} =\{X\}, \hat{\lan}=\{Y\}$ and the order $R>G$. Case \#27 satisfies Condition \ref{condition2} with $\lan ^{\star} =\{X,Y\}, \hat{\lan}=\emptyset$ and the order $G>R$.
\end{lem}

In the next two lemmas, we show that for a set of forbidden triangles $S$ satisfying either Condition \ref{condition1} or Condition \ref{condition2}, the corresponding $Forb_c(S)$ forms a prioritised semi-free amalgamation class as defined in Definition \ref{priority}. Hence, by Theorem \ref{sir}, there is a stationary independence relation on the corresponding $\M_S$.

\begin{lem}\label{inforb}
Let $S$ be a set of forbidden triangles satisfying Condition \ref{condition1}. For any $A,B,C \in Forb_c(S)$ such that $B \subseteq A,C$, let $A \otimes _B C$ be the amalgamation defined in Definition \ref{priority}. Then $A \otimes _B  C \in Forb_c(S)$.
\end{lem}
\begin{proof}
Suppose there is a forbidden triangle in $A \otimes _B C$. Then it is either of the form $acc'$ for some $a \in A \setminus B, c,c' \in C \setminus B$ or $aa'c$ for some $a,a' \in A \setminus B, c\in C \setminus B$.

Without loss of generality, we may assume $aa'c$ is a forbidden triangle in $A \otimes _B C$. Since $(a,c)$, $(a',c)$ are amalgamated by relations from $\lan '$, $\rel (a,c) ,\rel (a',c)  \in \lan '$. Then $\rel (a,c)\rel (a',c)\rel (a,a')$ is of the form $R_i R_j R' \in S$ where $R_i,R_j \in \lan '$ and $R' \in \lan$, which contradicts the condition \ref{condition1}. Therefore, we have $A \otimes _B C \in Forb_c(S)$.
\end{proof}

\begin{lem}
Let $S$ be a set of forbidden triangles satisfying Condition \ref{condition2}. For any $A,B,C \in Forb_c(S)$ such that $B \subseteq A,C$, let $A \otimes _B C$ be the amalgamation defined in Definition \ref{priority}. Then $A \otimes _B  C \in Forb_c(S)$.
\end{lem}

\begin{proof}

Suppose there exists a forbidden triangle in $AC \setminus B$. As in the previous lemma, we may assume it is $aa'c$ for some $a,a' \in A \setminus B, c\in C \setminus B$ and $\rel (a,c) ,\rel (a',c)  \in \lan '$. Since the only forbidden triangle involving $R_i R_j R$ where $R_i,R_j \in \lan '$ and $R\in \lan$ is $R_2R_2R^1$ where $R ^1\in \lan ^{\star}$, we have $\rel (a,c)= \rel (a',c)= R_2$ and $\rel (a,a') \in \lan^{\star}$. This implies that there exists $b\in B$ that forbids $(a,c)$ to be completed by $R_1$.

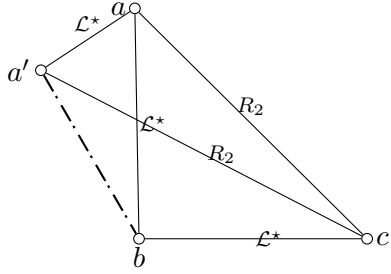
\begin{center}
 \begin{tikzpicture}
    \tikzstyle{every node}=[draw,circle,fill=white,minimum size=4pt,
                            inner sep=0pt]
  \draw (0,0) node (1) [label=left:$a$] {}
     ++(315:4.32cm) node (2) [label=right:$c$]{} 
   ++(180:3cm) node (3)  [label=below:$b$] {}
   ++(120:2.58cm) node (4) [label=left:$a'$] {};

\draw  (4) -- (2)  node[draw=none,fill=none,font=\scriptsize,midway,right] {$R_2$} ;
\draw  (1) -- (2) node[draw=none,fill=none,font=\scriptsize,midway,above] {$R_2$} ;
\draw (1) -- (3)  node[draw=none,fill=none,font=\scriptsize,midway,right] {$\lan^{\star}$};
\draw (2) -- (3)  node[draw=none,fill=none,font=\scriptsize,midway,right] {$\lan^{\star}$};
\draw [thick,dash pattern={on 7pt off 2pt on 1pt off 3pt}] (4) -- (3)  node[draw=none,fill=none,font=\scriptsize,midway,right] {};
\draw (1) -- (4)  node[draw=none,fill=none,font=\scriptsize,midway,above] {$\lan^{\star}$};
\end{tikzpicture}
\end{center}

Since the only forbidden triangle containing $R_1$ is of the form $R_1R^1R^2$ where $R^1,R^2 \in \lan ^{\star}$, we have $ \rel (a,b), \rel  (b,c) \in \lan ^{\star}$. 

As $S$ contains all triangles of the form $R'R^1R^2$ where $  R^1,R^2 \in \lan ^{\star}, R' \in \lan \setminus (\hat{\lan} \cup \{R_2\})$, this forces $\rel (a',b)$ to be either $R_2$ or in $\hat{\lan}$ as $\rel (a,a'), \rel (a,b)$ $ \in \lan^{\star}$. If $\rel (a',b) \in \hat{\lan}$, $a'bc$ forms a forbidden triangle of the form $R_2 \hat{R} R^1$, where $R^1 \in \lan ^{\star}, \hat{R} \in \hat{\lan}$. If $\rel (a',b)=R_2$, $a'bc$ forms a forbidden triangle of the form $R_2 R_2 R^1$, where $R^1 \in \lan ^{\star}$. Hence, we cannot find a colour for $a'b$ without creating a forbidden triangle. Thus, $A \otimes _B C \in Forb_c(S)$.
\end{proof}


\begin{rmk}\label{conditionstar}

\begin{enumerate}

\item It can be seen in the proofs of the previous two lemmas that when $S$ satisfies Condition \ref{condition1}, the order on $\lan'$ does not affect whether $A \otimes _B C \in Forb_c (S)$. However for $S$ satisfying Condition \ref{condition2}, whether $A \otimes _B C \in Forb_c (S)$ is dependent on $R_1 >R_2$.

\item For $S$ satisfying Condition \ref{condition1} or Condition \ref{condition2}, $S$ does not contain triangle involving $R_1R_1$ or $R_1R_2$, hence it satisfies part (i) of Condition \ref{maincond}. We will show part (ii) of Condition \ref{maincond} in the next lemma.

\end{enumerate}

\end{rmk}

\begin{lem}\label{delete}
For $S$ satisfying Condition \ref{condition1} or Condition \ref{condition2}, let $a,b, c \in \M_S$ and $B \subseteq \M_S$ such that $a \ind_{bB} c$. If $\rel (a,b) \in \lan'$, we have $a \ind _B c$.
\end{lem}
\begin{proof}
If $b \in B$, the statement is trivial. Let $b \notin B$. Since $a \ind _{bB} c$, we have $abBc=a \otimes _{bB} c$. To show that $a \ind_B c$, we want to show that  $(a,c)$ is coloured by the same relation in both $a \otimes _B c$ and $a \otimes _{bB} c$. Suppose on the contrary, $(a,c) \in R_i$ in $a \otimes _B c$ for some $R_i \in \lan'$ and $(a,c) \notin R_i$ in $a \otimes _{bB} c$, this implies that $R_i\rel (a,b) \rel (b,c)$ forms a forbidden triangle.

For $S$ satisfying Condition \ref{condition1}, since $\rel (a,b), R_i \in \lan '$, $R_i \rel (a,b) \rel (b,c)$ is of the form $R_iR_jR'$ where $R_i R_j \in \lan ', R' \in \lan$, contradicting Condition \ref{condition1}.

For $S$ satisfying Condition \ref{condition2}, since $\lan '=\{R_1,R_2\}$ with $R_1>R_2$, the only possibility of $(a,c) \notin R_i$ in $a\otimes _{bB} c$ is when $i=1$. Then $R_1 \rel (a,b) \rel (b,c)$ is a forbidden triangle in $S$. Since $\rel (a,b) \in \lan'$, $\rel(a,b) =R_1$ or $R_2$. However, $S$ does not contain any triangle involving $R_1R_1$ or $R_1R_2$ as noted in Remark \ref{conditionstar}.

Therefore, given the assumption, it can be deduced from $a \ind_{bB} c$ that $a \ind_{B} c$.
\end{proof}

Therefore, applying Theorem \ref{movealmost}, we obtain the following theorem:
\begin{thm}
Let $S$ be a set of forbidden triangles satisfying Condition \ref{condition1} or Condition \ref{condition2} and $\M_S$ be the Fra{\"\i}ss{\'e} limit of $Forb_c(S)$. Then for any non-trivial automorphism $g\in Aut(\M_S)$, there exist $k,h \in Aut(\M_S)$ such that $[k,[h,g]]$ moves almost maximally. Hence, $Aut(\M_S)$ is simple.
\end{thm}

\subsection{The Remaining Case}

The remaining case is the following:

$\lan =\{ R,G,X,Y \}$

$\# 26$ \hspace{0.8cm} RRX  RXX  RYY  GYX GXX  YYX  XXX 

\vspace{0.3cm}

We will first show that in case $\# 26$, with $\lan'=\{ R,G,Y\}$, $Forb_c(S)$ forms a prioritised semi-free amalgamation class. In fact, $\{ R,G,Y\}$ is the only possible set of solutions for $Forb_c(S)$ to be a semi-free amalgamation class. Hence, we can find a stationary independence relation on $\M_S$ and finally we will prove the simplicity of the automorphism group of $\M_S$ with a similar approach, but with some extra conditions.

\begin{lem}
Let $S$ be as in $\# 26$ with the order $G>R>Y$ on $\lan'=\{ R,G,Y \}$. For any $A,B,C \in Forb_c(S)$ such that $B \subseteq A,C$, let $A \otimes _B C$ be the prioritised semi-free amalgamation defined in Definition \ref{priority}. Then $A \otimes _B  C \in Forb_c(S)$.
\end{lem}
\begin{proof}
Suppose there exists a forbidden triangle in $A \otimes _B C$. As a forbidden triangle can only appear in $AC \setminus B$ by construction, we may assume $aa'c$ is a forbidden triangle in $A \otimes _B C$ for some $a,a' \in A \setminus B, c \in C \setminus B$. Since $\rel (a,c), \rel (a',c) \in \lan'$, these are not $X$ and hence, there are five possible cases for  $\rel (a,c) \rel (a',c)\rel (a,a')$, as listed below. In each case, since $\rel (a,c) \neq G$, there exists $b\in B$ that forbids $(a,c)$ to be coloured by $G$. $(a',b)$ is the edge where we cannot find any relation and hence we have a contradiction.

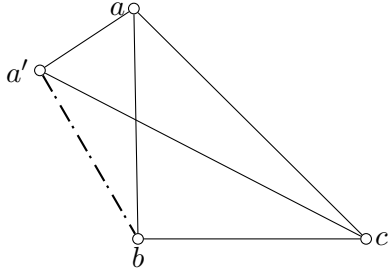
\begin{center}
 \begin{tikzpicture}
    \tikzstyle{every node}=[draw,circle,fill=white,minimum size=4pt,
                            inner sep=0pt]
  \draw (0,0) node (1) [label=left:$a$] {}
     ++(315:4.32cm) node (2) [label=right:$c$]{} 
   ++(180:3cm) node (3)  [label=below:$b$] {}
   ++(120:2.58cm) node (4) [label=left:$a'$] {};

\draw (4) -- (2)  node[draw=none,fill=none,font=\scriptsize,midway,right]  {};
\draw (1) -- (2) node[draw=none,fill=none,font=\scriptsize,midway,above] {} ;
\draw (1) -- (3)  node[draw=none,fill=none,font=\scriptsize,midway,right] {};
\draw (2) -- (3)  node[draw=none,fill=none,font=\scriptsize,midway,below] {};
\draw [thick,dash pattern={on 7pt off 2pt on 1pt off 3pt}] (4) -- (3)  node[draw=none,fill=none,font=\scriptsize,midway,right] {};
\draw (1) -- (4)  node[draw=none,fill=none,font=\scriptsize,midway,above] {};
\end{tikzpicture}
\end{center}

\begin{enumerate}[I.]
  \item $\rel (a,c) \rel (a',c)\rel (a,a')=RRX$.

Since $\rel (a,c) =R$, either $\rel (b,a) \rel (b,c)=YX$ or $\rel (b,a) \rel (b,c)=XY$. However, in the first case, $\rel (a,a') \rel (a,b)=XY$ forbids $\rel( a',b)$ to be $Y,G,X$ and $\rel (c,a') \rel (c,b)=RX$ forbids $\rel( a',b)$ to be $R,X$. Therefore, we cannot find a colour for $(a',b)$ without creating a forbidden triangle. In the second case, $\rel (a,a') \rel (a,b)=XX$ forbids $\rel( a',b)$ to be $R,G,X$ and $\rel (c,a') \rel (c,b)=RY$ forbids $\rel( a',b)$ to be $Y$. Hence, in both case, there would exist a forbidden triangle.

  \item $\rel (a,c) \rel (a',c)\rel (a,a')=YYR$

Since $\rel(a,c)=Y$, $\rel (b,a) \rel (b,c)=XX$. However, $\rel (a,a') \rel (a,b)=RX$ forbids $\rel( a',b)$ to be $R,X$ and $\rel (c,a') \rel (c,b)=YX$ forbids $\rel( a',b)$ to be $Y,G,X$, a contradiction.

  \item $\rel (a,c) \rel (a',c)\rel (a,a')=YYX$

Similarly as in the previous cases, $\rel (b,a) \rel (b,c)=XX$. However, $\rel (a,a') \rel (a,b)=XX$ forbids $\rel( a',b)$ to be $G,R,X$ and $\rel (c,a') \rel (c,b)=YX$ forbids $\rel( a',b)$ to be $Y,G,X$, a contradiction.

 \item $\rel (a,c) \rel (a',c)\rel (a,a')=YRY$
 We have $\rel (b,a) \rel (b,c)=XX$. However, $\rel (a,a') \rel (a,b)=YX$ forbids $\rel( a',b)$ to be $G,Y$ and $\rel (c,a') \rel (c,b)=RX$ forbids $\rel( a',b)$ to be $R,X$, a contradiction.

  \item $\rel (a,c) \rel (a',c)\rel (a,a')=YGX$

In this case, we have $\rel (a,b) \rel (b,c)=XX$. However, $\rel (a,a') \rel (a, b)=XX$ forbids $\rel (a',b)$ to be $R,G,X$ and $\rel (a',c) \rel (c, b)=GX$ forbids $\rel (a',b)$ to be $Y$, a contradiction.

\end{enumerate}

\end{proof}

As noted earlier, there appears to be a claim in \cite{cherlin1998classification} that a solution with 2 colours gives a semi-free amalgamation class for $\# 26$. This is incorrect as we can prove the following:

\begin{lem}\label{no26semifree}

Let $S$ be as in $\# 26$ and $Forb_c(S)$ be the set of all complete $\lan$-structures that do not embed any triangle in $S$. Then $\{ R,G,Y\}$ is the only possible set of soltions for $Forb_c(S)$ to be a semi-free amalgamation class.
\end{lem}
\begin{proof}
Since in the previous lemma, we showed that with $\lan '= \{ R,G,Y\}$, $Forb_c(S)$ forms a prioritised semi-free amalgamation class, it remains to show that $\lan'$ has to contain $G,R,Y$ by finding amalgamations where each of them is the only choice.

In the following amalgamation, $(a,c)$ has to be coloured by $G$ as $YY$ forbids it to be coloured $R$ or $X$ and $GX$ forbids it to be coloured $Y$. Hence $G \in \lan'$.
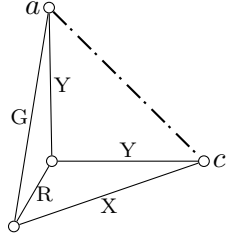
\begin{center}
 \begin{tikzpicture}
    \tikzstyle{every node}=[draw,circle,fill=white,minimum size=4pt,
                            inner sep=0pt]
  \draw (0,0) node (1) [label=left:$a$] {}
     ++(315:2.88cm) node (2) [label=right:$c$]{} 
   ++(180:2cm) node (3) {}
     ++ (240:1cm) node (4) {};

\draw [thick,dash pattern={on 7pt off 2pt on 1pt off 3pt}] (1) -- (2) ;
\draw (1) -- (3)  node[draw=none,fill=none,font=\scriptsize,midway,right] {Y};
\draw (2) -- (3)  node[draw=none,fill=none,font=\scriptsize,midway,above] {Y};
\draw (4) -- (3)  node[draw=none,fill=none,font=\scriptsize,midway,right] {R};
\draw (1) -- (4)  node[draw=none,fill=none,font=\scriptsize,midway,left] {G};
\draw (2) -- (4)  node[draw=none,fill=none,font=\scriptsize,midway,below] {X};

\end{tikzpicture}
\end{center}

Similarly, in the following amalgamation, $(a,c) \in R$ as $XY$ forbids it to be coloured $G$ or $Y$ and $RR$ forbids it to be coloured $X$. Hence $R \in \lan'$
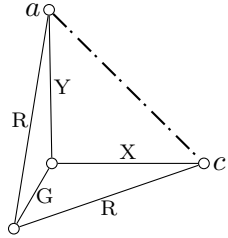
\begin{center}
 \begin{tikzpicture}
    \tikzstyle{every node}=[draw,circle,fill=white,minimum size=4pt,
                            inner sep=0pt]
  \draw (0,0) node (1) [label=left:$a$] {}
     ++(315:2.88cm) node (2) [label=right:$c$]{} 
   ++(180:2cm) node (3) {}
     ++ (240:1cm) node (4) {};

\draw [thick,dash pattern={on 7pt off 2pt on 1pt off 3pt}] (1) -- (2) ;
\draw (1) -- (3)  node[draw=none,fill=none,font=\scriptsize,midway,right] {Y};
\draw (2) -- (3)  node[draw=none,fill=none,font=\scriptsize,midway,above] {X};
\draw (4) -- (3)  node[draw=none,fill=none,font=\scriptsize,midway,right] {G};
\draw (1) -- (4)  node[draw=none,fill=none,font=\scriptsize,midway,left] {R};
\draw (2) -- (4)  node[draw=none,fill=none,font=\scriptsize,midway,below] {R};

\end{tikzpicture}
\end{center}

The following amalgamation implies $Y \in \lan'$ since $XX$ forbids $(a,c)$ to be coloured $R$, $G$ or $X$. Hence $Y \in \lan'$.
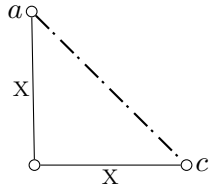
\begin{center}
 \begin{tikzpicture}
    \tikzstyle{every node}=[draw,circle,fill=white,minimum size=4pt,
                            inner sep=0pt]
  \draw (0,0) node (1) [label=left:$a$] {}
     ++(315:2.88cm) node (2) [label=right:$c$]{} 
   ++(180:2cm) node (3) {};

\draw [thick,dash pattern={on 7pt off 2pt on 1pt off 3pt}] (1) -- (2) ;
\draw (1) -- (3)  node[draw=none,fill=none,font=\scriptsize,midway,left] {X};
\draw (2) -- (3)  node[draw=none,fill=none,font=\scriptsize,midway,below] {X};

\end{tikzpicture}
\end{center}

Therefore, we have shown that $\{ R, G, Y\} \subseteq \lan'$. 

\end{proof}

\begin{rmk}\label{deletenotholdfor26}
With order $G>R>Y$ on $\lan' =\{ G,R,Y\}$, part (i) of Condition \ref{maincond} holds, but part (ii) does not as in the following example, we have $a \ind _{b b'} c $ and $\rel (a,b) \in \lan'$, but $a$ is not independent from $c$ over $b'$.
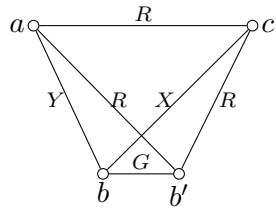
\begin{center}
 \begin{tikzpicture}
    \tikzstyle{every node}=[draw,circle,fill=white,minimum size=4pt,
                            inner sep=0pt]
  \draw (0,0) node (1) [label=left:$a$] {}
     ++(0:2.88cm) node (2) [label=right:$c$]{} 
   ++(225:2.78cm) node (3) [label=below:$b$] {}
     ++ (0:1cm) node (4) [label=below:$b'$] {};

\draw (1) -- (2)  node[draw=none,fill=none,font=\scriptsize,midway,above] {$R$} ;
\draw (1) -- (3)  node[draw=none,fill=none,font=\scriptsize,midway,left] {$Y$};
\draw (2) -- (3)  node[draw=none,fill=none,font=\scriptsize,midway,left] {$X$};
\draw (4) -- (3)  node[draw=none,fill=none,font=\scriptsize,midway,above] {$G$};
\draw (1) -- (4)  node[draw=none,fill=none,font=\scriptsize,midway,right] {$R$};
\draw (2) -- (4)  node[draw=none,fill=none,font=\scriptsize,midway,right] {$R$};
\end{tikzpicture}
\end{center}
However, we can prove a similar version, namely, for $a \ind _{bB} c$ where both $\rel (a,b), \rel (b,c) \in \lan '$, we have $a \ind _B c$, proved in the following lemma. We will then show that the proof of our main result, Corollary \ref{maincoro} follows similarly.
\end{rmk}

\begin{lem}\label{deletefor26}
Let $S$ be as in $\# 26$. Let $a,b,c \in \M_S$ and $B \subseteq \M_S$ such that $\rel (a,b), \rel (b,c) \in \lan'$ and $a \ind_{bB} c$. Then $a \ind _B c$.
\end{lem}
\begin{proof}
If $b\in B$, the statement is trivial. Let $b \notin B$. As in Lemma \ref{delete}, to show that $a\ind_B c$, we want to show that $(a,c)$ is coloured by the same relation in both $aB \otimes _B cB$ and $abB \otimes _{bB} cbB$.

If $\rel (a,c)=G$ in $aB \otimes _B cB$ and $\rel (a,c) \neq G$ in $aB \otimes _{bB} cB$, then $G \rel (a,b) \rel (b,c)$ forms a forbidden triangle where $\rel(a,b), \rel (b,c) \in \lan'$, i.e. $\rel (a,b), \rel(b,c) \neq X$, by assumption. However, all forbidden triangles in $S$ involving $G$ also contains $X$, a contradiction.

If $\rel (a,c)=R$ in $aB \otimes _B cB$, then there exists $b' \in B$ such that $b'$ forbids $(a,c)$ to be coloured by $G$. Since $XX$ forbids both $G$ and $R$, $\rel (a,b') \rel (b',c) =XY$ or $YX$. However, $XY$ forbids $(a,c)$ to be coloured by $G$ or $Y$ and hence makes $R$ the only choice for $\rel (a,c)$ in $aB \otimes _{bB} cB$. The same argument works for when $\rel (a,c)=Y$ since $XX$ forbids both $G,R$.

Therefore, given the assumption, it can be deduced from $a \ind_{bB} c$ that $a \ind_{B} c$.
\end{proof}


Note that only Theorem \ref{setequal} and \ref{movealmost} depend on part (ii) of Condition \ref{maincond}. Theorem \ref{setequal} holds for \# 26 since in the proof, $tp(c/c'B)=tp(c''/c'B)$ and then we have $\rel (c,c')=\rel(c'',c') \in \lan'$. So, we again have $a \ind _B c$ from $a \ind _{bB} c$ by Lemma \ref{deletefor26} and the rest of the proof follows. 

For Theorem \ref{movealmost}, let $g \in Aut(\M_S)$ be a non-trivial automorphism of $\M_S$. As Lemma \ref{colourrange} only depends on $Forb_c(S)$ being a semi-free amalgamation class, we can construct $h\in Aut(\M_S)$ such that for any non-algebraic 1-type $p$ over some finite set, there exist infinitely many realisations $a$ of $p$ such that $\rel (a, [h,g]a) \in \lan '$. We want to show that there exists $k \in Aut(\M_S)$ such that $[k,[h,g]]$ moves almost maximally. The proof is mostly the same as the proof of Theorem \ref{movealmost} except in Equation (\ref{bindinv}), we choose $b$ also satisfying $\rel (b,gb) \in \lan'$. Then $\rel (a,c) =\rel (ka, kc) =\rel (b,gb) \in \lan'$. Thus, by the previous lemma, we again have $c \ind _A ga$ from $c \ind_{aA} ga$. The rest of the proof of Theorem \ref{movealmost} follows, which proves the following corollary:

\begin{coro}
Let $S$ be as in $\# 26$ and $\M_S$ be the Fra{\"\i}ss{\'e} limit of $Forb_c(S)$. Then for any non-trivial automorphism $g\in Aut(\M_S)$, there exist $k,h \in Aut(\M_S)$ such that $[k,[h,g]]$ moves almost maximally. Hence, $Aut(\M_S)$ is simple.
\end{coro}

Combining the results in this section, we then have shown Theorem \ref{maincherlin}.
\vspace{1cm}

\printbibliography

\end{document}